\documentclass[preprint,12pt]{elsarticle}
\usepackage{graphicx}

\usepackage{epsfig}
\usepackage[cp1251]{inputenc}
\usepackage{amssymb}
\usepackage{amsthm}
\usepackage{cmap}

\usepackage{amsmath}% http://ctan.org/pkg/amsmath

\biboptions{sort&compress}

\newcommand{\sysn}{\left\{\begin{array}{rcl}}
\newcommand{\sysk}{\end{array}\right.}

\newtheorem{theorem}{Theorem}[section]
\newtheorem{lemma}[theorem]{Lemma}

\theoremstyle{example}

\newtheorem{proposition}[theorem]{Proposition}

\newtheorem{corollary}[theorem]{Corollary}
\theoremstyle{definition}
\newtheorem{definition}[theorem]{Definition}

\newtheorem{remark}[theorem]{Remark}
\journal{...}

\begin{document}

\title{Tightness type properties of spaces of quasicontinuous functions}

\author{Anton E. Lipin, Alexander V. Osipov}

\address{Krasovskii Institute of Mathematics and Mechanics, \\ Ural Federal
 University, Yekaterinburg, Russia}

\ead{tony.lipin@yandex.ru, oab@list.ru}

\begin{abstract} Using approximation by continuous functions we prove the following statements to types of tightness in a space $Q_p(X, \mathbb{R})$ of all quasicontinuous real-valued functions with the topology $\tau_p$ of pointwise convergence: the countability of tightness (fan-tightness, strong fan-tightness)  at a point $f$ of space $Q_p(X, \mathbb{R})$  implies the countability of tightness (fan-tightness, strong fan-tightness) of space $Q_p(X,Y)$ of all quasicontinuous functions from $X$ into any non-one-point metrizable space $Y$. This result is the answer to the open question in the class of metrizable spaces.

\end{abstract}

%\tnotetext[label1]{The research has been supported by .}


\maketitle %%
% \linenumbers


\section{Introduction}

A function $f:X\rightarrow Y$ is {\it quasicontinuous} at $x$ if for any open set $V$ containing $f(x)$ and any $U$ open containing $x$, there exists a nonempty open set
$W\subseteq U$ such that $f(W)\subseteq V$. It is {\it quasicontinuous} if it is quasicontinuous at
every point.
Call a set semi-open (or quasi-open) if it is contained in the closure of its
interior. Then $f: X\rightarrow Y$ is quasicontinuous if and only if the inverse of every
open set is quasi-open.

\medskip

Let $X$ and $Y$ be Hausdorff  topological spaces, $Q_p(X, Y)=(Q(X, Y),\tau_p)$ be
the space of all  quasicontinuous functions on $X$ with values in
$Y$ and $\tau_p$ be the pointwise convergence topology.

%\begin{definition} Suppose $X$ and $Y$ are topological spaces. A set $U \subseteq X$ is called {\it quasi-open}, if it is contained in the closure of its interior. A function $f:X \to Y$ is called {\it quasi-continuous}, if for every open set $V \subseteq Y$ the set $f^{-1}(V)$ is quasi-open in $X$. The space of all quasi-continuous functions $f:X \to Y$ with the topology of pointwise convergence is denoted by $Q_p(X,Y)$.
%\end{definition}

We were motivated for this paper by the problem of tightness in a space of quasi-continuous functions.

It is an open problem whether the space $Q_p(X,\mathbb{R})$ is topologically homogeneous (note that this space is not a linear space, since the sum of quasi-continuous functions is not necessarily quasi-continuous). In particular, it is unknown whether the points of $Q_p(X,\mathbb{R})$ can be distinguished using some topological property or cardinal invariant. In this regard, in (\cite{Osipov}, Question 1) the second author raised a question on tightness in the space $Q_p(X,\mathbb{R})$. It was proved in \cite{Osipov} that for every continuous $f : X \to \mathbb{R}$ we have $t(f, Q_p(X,\mathbb{R})) = \omega$ if and only if the space $X$ is $\mathcal{K}_\Omega$-Lindelof, and it was unknown whether the same is true for non-continuous functions.

In this paper we prove that the property of countable tightness (as well as countable fan-tightness and countable strong fan-tightness) can not distinguish points of the space $Q_p(X,Y)$ whenever $Y$ is metrizable.

One of the main tools of our proof is Theorem \ref{T1} below on approximation of arbitrary functions by continuous functions (in the metric of uniform convergence).

\section{Preliminaries}

Let us recall some properties of a
topological space $X$.

(1)  A space $X$ has {\it countable tightness} at a point $x$ (denoted $t(x, X)=\omega$) if $x\in
\overline{A}$, then $x\in \overline{B}$ for some countable $B\subseteq A$. A space $X$ has countable tightness (denoted $t(X)=\omega$) if $t(x, X)=\omega$ for every $x\in X$.

(2) A space $X$ has {\it countable fan-tightness} at a point $x$ (denoted $vet(x, X)=\omega$)  if for any countable family $\{A_n : n\in\omega\}$
of subsets of $X$ satisfying $x\in \bigcap_{n\in \omega} \overline{A_n}$ it is possible to select finite sets $K_n\subset A_n$ in such a way that $x\in \overline{\bigcup_{n\in\omega} K_n}$.
A space $X$ has countable fan-tightness (denoted $vet(X)=\omega$) if $vet(x, X)=\omega$ for every $x\in X$.

(3)  A space $X$ is said to have {\it countable
strong fan-tightness} at a point $x$ (denoted $vet_1(x, X)=\omega$) if for each countable family $\{A_n: n\in \omega\}$
of subsets of $X$ such that $x\in \bigcap_{n\in \omega} \overline{A_n}$, there exist $a_i\in A_i$ such that
$x\in \overline{\{a_i: i\in\omega\}}$. A space $X$ has countable strong fan-tightness (denoted $vet_1(X)=\omega$) if $vet_1(x, X)=\omega$ for every $x\in X$.

(4) A space $X$ is said to be  {\it open
Whyburn} if for every open set $A\subset X$ and every $x\in
\overline{A}\setminus A$ there is an open set $B\subseteq A$ such
that $\overline{B}\setminus A=\{x\}$ \cite{Os1}.

\medskip
Note that the class of open
Whyburn spaces is quite wide; for example, it includes all first countable regular spaces \cite{Os1} and, therefore, all metrizable spaces.

\medskip

Let $X$ be a Tychonoff topological space, $C(X,\mathbb{R})$ be the
space of all  continuous functions on $X$ with values in
$\mathbb{R}$ and $\tau_p$ be the pointwise convergence topology.
Denote by $C_p(X,\mathbb{R})$ the topological space
$(C(X,\mathbb{R}), \tau_p)$.

 The symbol $\bf{0}$
stands for the constant function to $0$. A basic open neighborhood
of $\bf{0}$ in $\mathbb{R}^X$  is of the form $[F, (-\varepsilon, \varepsilon)]=\{f\in
\mathbb{R}^X: f(F)\subset (-\varepsilon, \varepsilon)\}$, where $F\in
[X]^{<\omega}$ and $\varepsilon>0$.

\medskip
Let us recall that a cover $\mathcal{U}$ of a set $X$ is called

$\bullet$ an {\it $\omega$-cover} if each finite set $F\subseteq
X$ is contained in some $U\in \mathcal{U}$;

$\bullet$ a {\it $\gamma$-cover} if for any $x\in X$ the set
$\{U\in \mathcal{U}: x\not\in U\}$ is finite.

\medskip

In this paper $\mathcal{A}$ and $\mathcal{B}$ will be collections
of the following  covers of a space $X$:

%$\mathcal{O}$ : the collection of all open covers of $X$.

$\mathcal{O}^s$ : the collection of all semi-open covers of $X$.

$\Omega$ : the collection of open $\omega$-covers of $X$.

$\mathcal{K}$: the collection $\mathcal{U}$ of open subsets of $X$ such that $X=\bigcup\{\overline{U}: U\in \mathcal{U}\}$.

$\mathcal{K}_{\Omega}$ is the set of $\mathcal{U}$ in $\mathcal{K}$ such that no element of $\mathcal{U}$ is dense in $X$, and for each finite set $F\subseteq X$, there is a $U\in \mathcal{U}$ such that $F\subseteq \overline{U}$.

 \medskip

A space is {\it $\mathcal{K}_{\Omega}$-Lindel\"{o}f} if each element of $\mathcal{K}_{\Omega}$ has a countable subset in $\mathcal{K}_{\Omega}$ \cite{Sh2}.

\begin{definition}(\cite{Kun})
A Hausdorff space $X$ is called a {\it Lusin space}  if

(a) Every nowhere dense set in $X$ is countable;

(b) $X$ has at most countably many isolated points;

(c) $X$ is uncountable.
\end{definition}

If $X$ is an uncountable Hausdorff space then $X$ is $\mathcal{O}^s$-Lindel\"{o}f (semi-Lindel\"{o}f) if
and only if $X$ is a Lusin space (Corollary 2.5 in \cite{GJR}).

\medskip

If $X$ is a Lusin space, $X$ is hereditarily
Lindel\"{o}f (Lemma 1.2 in~\cite{Kun}). Hence, if $X$ is a regular Lusin space then $X$ is perfect normal (3.8.A.(b)~in~\cite{Eng}).

\medskip

If $X$ is a Lusin space, so is every uncountable subspace (Lemma 1.1 in~\cite{Kun}).

\medskip

Many topological properties are defined or characterized in terms
 of the following classical selection principles (see \cite{H1}).
 Let $\mathcal{A}$ and $\mathcal{B}$ be sets consisting of
families of subsets of an infinite set $X$. Then:

$S_{1}(\mathcal{A},\mathcal{B})$ is the selection hypothesis: for
each sequence $\{A_{n}: n\in \mathbb{N}\}$ of elements of
$\mathcal{A}$ there is a sequence $\{b_{n}\}_{n\in \mathbb{N}}$
such that for each $n$, $b_{n}\in A_{n}$, and $\{b_{n}:
n\in\mathbb{N} \}\in \mathcal{B}$.

$S_{fin}(\mathcal{A},\mathcal{B})$ is the selection hypothesis:
for each sequence $\{A_{n}: n\in \mathbb{N}\}$ of elements of
$\mathcal{A}$ there is a sequence $\{B_{n}\}_{n\in \mathbb{N}}$ of
finite sets such that for each $n$, $B_{n}\subseteq A_{n}$, and
$\bigcup_{n\in\mathbb{N}}B_{n}\in\mathcal{B}$.

\medskip
In this paper, by metric space we mean a non-trivial metric space, that is, of cardinality greater than 1.

\medskip
For other notation and terminology almost without exceptions we follow the Engelking's book \cite{Eng} and the papers \cite{Os1,Osipov}.

\section{Approximation by continuous functions}

\begin{theorem}\label{T1} Let $X$ be a normal space. For any $f\in \mathbb{R}^X$ there exists a continuous function $g:X\rightarrow \mathbb{R}$ such that $|f(x)-g(x)|\leq \frac{osc(f)}{2}$, where $osc(f)$ is the oscillation of $f$.
\end{theorem}

This result can be found, for instance, in  \cite{Ben} (Proposition 1.18) (see also \cite{Eng}, Problem 1.7.15 (b)).

%\begin{remark}
%It is obvious that the upper bound $\frac{osc(f)}{2}$ is the best possible.
%\end{remark}

Now let us denote $\mathbb{H}$ the Hilbert cube $\prod\limits_{n \in \mathbb{N}} [0, \frac{1}{n}]$ with the metric $d_\mathbb{H}$ defined as following:
$$d_\mathbb{H}(x, y) = \max\limits_{n \in \mathbb{N}} |x(n) - y(n)|.$$

It is easy to see that the metric $d_\mathbb{H}$ is equivalent to the most common euclidean metric of $\mathbb{H}$.
The known fact that the space $\mathbb{H}$ is universal for all separable metrizable spaces must explain our attention to the following

\begin{proposition}
%\label{C_A}
If $X$ is a normal space, then for every function $f:X \to \mathbb{H}$ there is a continuous function $g : X \to \mathbb{H}$ such that
$$d_\mathbb{H}(f(x), g(x)) \leq \frac{osc(f)}{2}$$
for all $x \in X$.
\end{proposition}
\begin{proof}
Suppose $f(x) = \{f_n(x)\}_{n \in \mathbb{N}}$.
For every $n \in \mathbb{N}$ choose some continuous function $g_n : X \to \mathbb{R}$ such that
$|f_n(x) - g_n(x)| \leq \frac{osc(f_n)}{2} \leq \frac{osc(f)}{2}$ for all $x \in X$.
Define $g(x) = \{g_n(x)\}_{n \in \mathbb{N}}$ for all $x \in X$.
It is easy to see that $g: X \to \mathbb{H}$ is continuous and $d_\mathbb{H}(f(x), g(x)) \leq \frac{osc(f)}{2}$.
\end{proof}

\section{Domain approximation}
\label{section_da}

\begin{definition}
Suppose we are given a topological space $X$, a metric space $Y$ and a function $f : X \to Y$. We say that a countable family $\mathcal{E}$ of closed subsets of $X$ is a $\mathrm{DA}${\it -family} (or {\it domain approximation family}) for the function $f$ if for every finite $K \subseteq X$ and any $\varepsilon > 0$ there is a set $F \in \mathcal{E}$ such that $K \subseteq F$ and $osc(f|_F) < \varepsilon$.
\end{definition}

\begin{definition}
Suppose we are given a topological space $X$, a metric space $Y$ and a function $f : X \to Y$.
For all $\varepsilon>0$ we denote $D_\varepsilon(f) = \{x \in X : osc(f,x) \geq \varepsilon\}$.
We also denote $D(f) = \bigcup\limits_{\varepsilon>0} D_\varepsilon(f)$, which is the set of all discontinueties of $f$.
\end{definition}

\begin{lemma}
Suppose we are given metric spaces $X$ and $Y$ and a function $f : X \to Y$ such that $|D(f)|\leq \omega$.
Then there is a $\mathrm{DA}$-family for $f$.
\end{lemma}
\begin{proof}
Let us denote $I=\mathbb{N} \times \mathbb{N} \times [D(f)]^{<\omega}$.
For all $\gamma=(n,k,M) \in I$ we define $$F_\gamma = M \cup X \setminus O_{1/k}(D_{1/n}(f)).$$
It is obvious that the family $\{F_\gamma : \gamma \in I\}$ is a $\mathrm{DA}$-family for $f$.
\end{proof}

\begin{corollary}\label{cor3.5}
If $X$ is a metrizable Lusin space, then for every metric space $Y$ and any $f \in Q(X,Y)$ there is a $\mathrm{DA}$-family for $f$.
\end{corollary}

\section{Flattening lemma}
\label{section_fl}

For every set $S$ we denote $(H_S, \rho_S)$ the metric $S$-hedgehog space, i.e. $H_S = \{0\} \cup ((0,1] \times S)$,
\begin{equation*}
\rho_S((p,\alpha), (q,\beta)) =
\begin{cases}
|p - q|, \; \alpha = \beta; \\
p + q, \; \alpha \ne \beta
\end{cases}
\end{equation*}
and $\rho_S(0, (p, \alpha)) = x$.

Recall that the $\omega$-th power of the space $H_S$ is an universal space for all metrizable spaces of weight $\leq|S|$ (see \cite{Eng}).
We fix on the space $H_S^\omega$ the metric
$$\rho_S^\omega(x, y) = \sup\limits_{n \in \omega} \frac{1}{n + 1}\rho_S(x(n), y(n)),$$
which generates the Tychonoff production topology.

\begin{definition}
Let $S$ be a set and suppose we are given $D \subseteq S$ and $\gamma \in D$.
We define the function $h_S^{D, \gamma} : H_S \to H_D$ in the following way: $h_S^{D, \gamma}(0) = 0$ and
\begin{equation*}
h_D^\gamma(x, \alpha) =
\begin{cases}
(x, \alpha), \; \alpha \in D; \\
(x, \gamma), \; \alpha \notin D.
\end{cases}
\end{equation*}
\end{definition}

Let us notice an obvious property of the function $h_D^\gamma$.

\begin{proposition}
Let $S$ be a set and suppose we are given $D \subseteq S$ and $\gamma \in D$.
Denote $C = D \setminus \{\gamma\}$.
Then for every $x \in H_C$ and $y \in H_S$ we have $\rho_S(h_D^\gamma(x), h_D^\gamma(y)) = \rho_S(x, y)$.
\end{proposition}

\begin{lemma}
\label{lemma_flattening}
Suppose we are given a metrizable space $X$ and its separable subset $A$.
Then there is a continuous function $\varphi : X \to \mathbb{H}$ such that:
\begin{itemize}
\item[$(*)$] $\varphi[U \cap A] \subseteq \mathrm{Int}_{\varphi[Y]}f[U]$ (for all open $U \subseteq Y$).
\end{itemize}
\end{lemma}
\begin{proof}
We can suppose that $X \subseteq H_S^\omega$ for some uncountable set $S$ \cite[Theorem 4.4.9]{Eng}.
Since $A$ is separable, we can choose a countable subset $C \subseteq S$ such that $A \subseteq H_C^\omega$.
Choose any $\gamma \in S \setminus C$ and denote $D = C \cup \{\gamma\}$.
Now for all $x \in X$ define $h(x) = \{h_D^\gamma(x(n))\}_{n \in \omega}$, so $h : X \to H_D^\omega$.
It follows from the previous propostion that $\rho_S^\omega(h(x), h(y)) = \rho_S^\omega(x, y)$ for all $x \in A$ and $y \in X$.
Finally, we take any homeomorphic embedding $g : H_D \to \mathbb{H}$ and define $\varphi(x) = g(h(x))$ for all $x \in X$.
The function $\varphi$ is as required.
\end{proof}

\begin{lemma}
\label{lemma_func_flattening}
Let $X$ be a space, $Y$ be a metrizable space and suppose we are given a function $f : X \to Y$ such that its image $A=f[X]$ is separable.
Take a function $\varphi : Y \to \mathbb{H}$ with the properties from Lemma \ref{lemma_flattening}.
Then for every family $P$ of functions $X \to Y$ the following conditions are equivalent:
\begin{enumerate}
\item[(1)] $f \in \overline{P}$;

\item[(2)] $f \circ \varphi \in \overline{\{h \circ \varphi : h \in P\}}$.
\end{enumerate}
\end{lemma}
\begin{proof}
$(1) \to (2)$ follows from continuity of $\varphi$.

$(2) \to (1)$ follows from the property (*).
\end{proof}

\section{Application to types of tightness of spaces of quasicontinuous functions}
\label{section_qcf}

\begin{definition}
Suppose we are given a topological space $X$, a metric space $(Y,d)$, a function $f : X \to Y$ and a familly $P \subseteq Y^X$.

\begin{itemize}

\item For every $r \geq 0$ we write $d(f, P) \leq r$ if for every finite $K \subseteq X$ there is a function $h \in P$ such that $d(f(x), h(x)) \leq r$ for all $x \in K$.

\item We denote $\textbf{d}(f,P) = \inf\{r \geq 0 : d(f,P) \leq r\}$.

\end{itemize}
\end{definition}

It is obvious that $f \in \overline{P}$ if and only if $\textbf{d}(f, P) = 0$.

\subsection{Tightness}

\begin{lemma}\label{lem6.2}
Let $X$ be an open
Whyburn topological space and $(Y,d)$ be metric space, $f\in C(X,Y)$ and suppose that $X$ is $\mathcal{K}_\Omega$-Lindel{\"o}f.
Then for every $P \subseteq Q(X,Y)$ such that $\textbf{d}(f,P)\leq \varepsilon$ there is a subset $S \in [P]^{\leq \omega}$ such that $\textbf{d}(f,S)\leq \varepsilon$.
\end{lemma}

\begin{proof} For every $h\in P$ and $n\in \mathbb{N}$ denote by $V_{h,n}=\{x\in X: d(f(x),h(x))<\varepsilon+\frac{1}{n}\}$. Note that $V_{h,n}$ is semi-open because $h\in Q(X,Y)$ and $V_{h,n}=\bigcup \{h^{-1}(\{y\in Y: d(f(x),y)<\varepsilon+\frac{1}{n}\}): x\in X\}$ (see Theorem 2 in \cite{Lev}).

Since $\textbf{d}(f,P)\leq \varepsilon$, $\mathcal{V}_n=\{V_{h,n}: h\in P\}$ is a semi-open $\omega$-cover of $X$ for every $n\in \mathbb{N}$.

Since $X$ is an open
Whyburn topological space, for every  $V_{h,n}$ and a finite subset $K$ of $V_{h,n}$ we can consider an open set $W_{K,h,n}$ such that $K\subset \overline{W_{K,h,n}}\subset V_{h,n}$.
Then the family $\{W_{K,h,n}: K\in [V_{h,n}]^{<\omega}$ and $h\in P\}\in \mathcal{K}_{\Omega}$ for every $n\in \mathbb{N}$.
For each $n\in \mathbb{N}$ we apply that $X$ is $\mathcal{K}_\Omega$-Lindel{\"o}f and we get the countable family $\{W_{K_i(n),h_i(n),n}: i\in \mathbb{N}\}\in \mathcal{K}_{\Omega}$.
It remains to note that $S=\{h_i(n): i,n\in \mathbb{N}\}$ such that $S\in [P]^{\leq \omega}$ and $\textbf{d}(f,S)\leq \varepsilon$.

\end{proof}

\begin{lemma}\label{lem5} Let $X$ be a metrizable $\mathcal{K}_\Omega$-Lindel{\"o}f space and $F$ be a closed subset of $X$. Then $F$ is $\mathcal{K}_\Omega$-Lindel{\"o}f.
\end{lemma}

\begin{proof} (1) Assume that $Int F=\emptyset$. Since $X$ is Lusin, $F$ is countable and, hence, $F$ is $\mathcal{K}_\Omega$-Lindel{\"o}f.

(2) Let $Int F\neq\emptyset$. By Theorem 3.6 ($(2)\Rightarrow (5)(a)$) in \cite{Osipov} (Any open subset of $\mathcal{K}_\Omega$-Lindel{\"o}f space is  $\mathcal{K}_\Omega$-Lindel{\"o}f),  $Int F$ is $\mathcal{K}_\Omega$-Lindel{\"o}f and, by Theorem 3.6 ($(2)\Rightarrow (5)(b)$) (The union of $\mathcal{K}_\Omega$-Lindel{\"o}f space with countable space is $\mathcal{K}_\Omega$-Lindel{\"o}f), $F=Int F\cup(F\setminus Int F)$ is $\mathcal{K}_\Omega$-Lindel{\"o}f.

\end{proof}

\begin{theorem}\label{th21}
For every metrizable space $X$ the following conditions are equivalent:
\begin{enumerate}
\item $t(Q_p(X,Y)) = \omega$ for all metrizable spaces $Y$;

\item There is a metrizable space $Y$ and a function $f \in Q_p(X, Y)$ such that $t(f, Q_p(X,Y))=\omega$;

\item There is a function $f \in Q_p(X, \mathbb{R})$ such that $t(f, Q_p(X,\mathbb{R}))=\omega$;

\item $X$ is $\mathcal{K}_\Omega$-Lindel{\"o}f.

\end{enumerate}
\end{theorem}

\begin{proof} $(1)\Rightarrow(2)$ and $(1)\Rightarrow(3)$  are trivial.

$(2)\Rightarrow(4)$.  (i) Suppose that $t(f, Q_p(X,Y))=\omega$ for some $f \in Q_p(X, Y)$, $f(X)\neq Y$ where $Y$ is metrizable and $|Y|>1$.
Let a family $\gamma=\{V_{\alpha}\}$ of open subsets of $X$ such that $\gamma\in \mathcal{K}_\Omega$. Let $y\in Y\setminus f(X)$.

For every $K\in [X]^{<\omega}$ there is $V_{\alpha}\in \gamma$ such that $K=\{x_1,...,x_k\}\subset \overline{V_{\alpha}}$.
Since $X$ is metrizable (hence, it is open
Whyburn), there are open subsets $W_1$,...,$W_k$ of $X$ such that $x_i\in \overline{W_i}\subset V_{\alpha}\cup\{x_i\}$  for $i=1,...,k$, $\overline{W_i}\cap \overline{W_j}=\emptyset$ for $i\neq j$.
Consider the function

 $$ f_{K,\alpha}:=\left\{
\begin{array}{lcl}
f(x_i) \, \, \, \, \, \, \, \, \, \, on \, \, \, \, \overline{W_i};\\
y \, \, \, \, \, \, \, \, on \, \, \, X\setminus (\bigcup\{\overline{W_i}: i=1,...,k\}).\\
\end{array}
\right.
$$

Let $F=\{f_{K,\alpha}: K\in [X]^{<\omega}\}$. It is clear that $f\in \overline{F}$. Then there is a countably subset $S=\{f_{K_i,\alpha_i}: i\in \mathbb{N}\}$ of $F$ such that $f\in \overline{S}$.

Claim that $\{V_{\alpha_i}: i\in \mathbb{N}\}\in \mathcal{K}_\Omega$. Let $K\in [X]^{<\omega}$. Consider the open neighborhood $[K,Y\setminus\{y\}]:=\{h\in Q_p(X,Y): h(x)\in Y\setminus\{y\}$ for every $x\in K\}$ of the point $f$ in the space $Q_p(X,Y)$. Then, there is $f_{K_i,\alpha_i}\in S\cap [K,Y\setminus\{y\}]$. It follows that $K\subset \overline{V_{\alpha_i}}$ and, hence, $\{V_{\alpha_i}: i\in \mathbb{N}\}\in \mathcal{K}_\Omega$.

(ii) Suppose that $t(f, Q_p(X,Y))=\omega$ for some $f \in Q_p(X, Y)$, $f(X)=Y$ where $Y$ is a metrizable space with a metric $d$.

\medskip
(1) Claim that there are non-empty open disjoint subsets $U$ and $V$ of $Y$ such that
$\overline{Int f^{-1}(U)\cup Int f^{-1}(V)}=X$.

Choose $p,q\in Y$ such that $p\neq q$ and for each $\varepsilon\in (0, d(p,q))$ we consider $U_{\varepsilon}=O_{\varepsilon}(p)$ and $V_{\varepsilon}=X\setminus \overline{O_{\varepsilon}(p)}$. It is clear that
$U_{\varepsilon}\cap V_{\varepsilon}=\emptyset$. Note that if $\overline{Int f^{-1}(U_{\varepsilon})\cup Int f^{-1}(V_{\varepsilon})}\neq X$ then $Int(f^{-1}[Fr O_{\varepsilon}(p)])\neq \emptyset$.
Let us show that for all $\varepsilon$ such a situation cannot exist.

Let's assume the opposite. For each $\varepsilon\in (0, d(p,q))$ we choose a point $x_{\varepsilon}\in Int(f^{-1}[Fr O_{\varepsilon}(p)])$ and a neighborhood $W_{\varepsilon}$ of $x_{\varepsilon}$ such that $\overline{W_{\varepsilon}}\subset Int(f^{-1}[Fr O_{\varepsilon}(p)])$. For each $E\subset (0, d(p,q))$ we define
the function $h_E$:

 $$ h_E:=\left\{
\begin{array}{lcl}
p \, \, \, \, \, \, \, \, \, \, on \, \, \, \, \bigcup\{\overline{W_{\varepsilon}}: \varepsilon\in E\};\\
f \, \, \, \, \, \, \, \, on \, \, \, X\setminus (\bigcup\{\overline{W_{\varepsilon}}: \varepsilon\in E\}).\\
\end{array}
\right.
$$

Let $Z=\{h_E: E\subset (0, d(p,q))\}$. Note that $Z\subset Q(X,Y)$ and $f=h_{\emptyset}\in Z$. It is clear that $Z$ is homeomorphic to $2^{\mathfrak{c}}$ and, hence, $t(f, Q_p(X,Y))> \omega$.

(2) Let $R=Y\setminus (U\cup V)$ where $U$ and $V$ from (1). For each $x\in f^{-1}(R)$ we choose $W_x\in \{U,V\}$ such that $f\upharpoonright (f^{-1}[W_x]\cup \{x\})$ is quasicontinuous. Let $A=f^{-1}(U)\cup \{x\in f^{-1}(R): W_x=U\}$ and $B=f^{-1}(V)\cup \{x\in f^{-1}(R): W_x=V\}$. Thus, $X=A\sqcup B$ and $f\upharpoonright A$ and $f\upharpoonright B$ are quasicontinuous.

Let $p\in U$ and $q\in V$.

(3) Let a family $\gamma=\{W_{\alpha}\}$ of open subsets of $X$ such that $\gamma\in \mathcal{K}_\Omega$.

We can represent each set $W_{\alpha}$ as a finite disjoint union $W^1_{\alpha}\sqcup ...\sqcup W^n_{\alpha}$ of open sets $W^j_{\alpha}$  where $W^j_{\alpha}\subset Int f^{-1}(U)$ or $W^j_{\alpha}\subset Int f^{-1}(V)$ for each $j=1,...,n$ and $|\overline{W^j_{\alpha}}\cap Fr A|\leq 1$. Note that since $X$ is an open
Whyburn space, we can assume this property.

For each $W_{\alpha}=W^1_{\alpha}\sqcup ...\sqcup W^n_{\alpha}\in \gamma$ and a finite subset $K=\{x_1,...,x_n\}$ of $X$ where $x_j\in \overline{W^j_{\alpha}}$, $j=1,...,n$, we define the function $h_{W_{\alpha}, K}$:

 $$ h_{W_{\alpha}, K}:=\left\{
\begin{array}{lcl}
f(x_j) \, \, \, \, \, \, \, \, \, \, on \, \, \, \, \overline{W^{j}_{\alpha}};\\
q \, \, \, \, \, \, \, \, on \, \, \, A\setminus W_{\alpha} \\
p \, \, \, \, \, \, \, \, on \, \, \, B\setminus W_{\alpha}. \\
\end{array}
\right.
$$

Note that $h_{W_{\alpha}, K}\in Q(X,Y)$ and $f\in \overline{\{h_{W_{\alpha}, K}: W_{\alpha}\in \gamma\}}$.

Since $t(f, Q_p(X,Y)) \leq \omega$, there is a countable family $\{h_{W_{\alpha_i}, K_i}: i\in \mathbb{N}\}$ such that $f\in \overline{\{h_{W_{\alpha_i}, K_i}: i\in \mathbb{N}\}}$.
It is remain that $\{W_{\alpha_i}: i\in \mathbb{N}\}\in \mathcal{K}_\Omega$.

$(4)\Rightarrow(1)$. By Lemma 3.1 in \cite{Osipov}, $X$ is a Lusin space.

Let $f\in Q_p(X,Y)$ and $f\in \overline{P}$ where $P\subset Q_p(X,Y)$.

Let us first reduce the general case to the case $Y=\mathbb{H}$.
Since $f$ is a quasicontinuous function on a Lusin space, the set $A=f[X]$ is separable.
Thus, by Lemma \ref{lemma_func_flattening} there is a function $\varphi : Y \to \mathbb{H}$ such that for every family $T$ of functions $X \to Y$ we have
$f \in \overline{T}$ if and only if $f \circ \varphi \in \overline{\{h \circ \varphi : h \in T\}}$, so now we only have to consider the case $Y=\mathbb{H}$.

By Corollary \ref{cor3.5}, there is a $\mathrm{DA}$-family $\mathcal{E}$ for $f$.
For every $F\in \mathcal{E}$ denote by $\varepsilon_F:=\mathrm{dc}(f, F)$. Thus, for every $F\in \mathcal{E}$ there is $g_F\in C(F,Y)$ such that $d(f(x),g_F(x))\leq \varepsilon_F$ for every $x\in F$.

By Lemma \ref{lem5}, $F$ is $\mathcal{K}_\Omega$-Lindel{\"o}f. Note that $d(g_F, P\upharpoonright F)\leq \varepsilon_F$.  By Lemma \ref{lem6.2}, there is $T_F\in [P]^{\leq \omega}$ such that $d(g_F,T\upharpoonright F)\leq \varepsilon_F$.

Consider $T=\bigcup \{T_F: F\in \mathcal{E}\}$. It is remain note that $T\in [P]^{\leq \omega}$  and $f\in \overline{T}$.

$(3)\Rightarrow(4)$. It is the implication $(2)\Rightarrow(4)$ for $Y=\mathbb{R}$.

\end{proof}

\subsection{Fan-tightness}

\begin{lemma}\label{lem6.2}
Let $X$ be an open
Whyburn topological space and $(Y,d)$ be metric space, $f\in C(X,Y)$ and suppose that $X$ satisfies $S_{\mathrm{fin}}(\mathcal{K}_\Omega, \mathcal{K}_\Omega)$.
Then for every $P_n \subseteq Q(X,Y)$ such that $\textbf{d}(f,P_n)\leq \varepsilon$ there is a finite subset $S_n$ of $P_n$ such that $\textbf{d}(f,\bigcup S_n)\leq \varepsilon$.
\end{lemma}

\begin{proof} For every $n\in \mathbb{N}$  and $h\in P_n$ and denote by $V_{h,n}=\{x\in X: d(f(x),h(x))<\varepsilon+\frac{1}{n}\}$. Note that $V_{h,n}$ is semi-open because $h\in Q(X,Y)$ and $V_{h,n}=\bigcup \{h^{-1}(\{y\in Y: d(f(x),y)<\varepsilon+\frac{1}{n}\}): x\in X\}$ (see Theorem 2 in \cite{Lev}).

Since $\textbf{d}(f,P)\leq \varepsilon$, $\mathcal{V}_n=\{V_{h,n}: h\in P\}$ is a semi-open $\omega$-cover of $X$ for every $n\in \mathbb{N}$.

Since $X$ is an open
Whyburn topological space, for every  $V_{h,n}$ and a finite subset $K$ of $V_{h,n}$ we can consider an open set $W_{K,h,n}$ such that $K\subset \overline{W_{K,h,n}}\subset V_{h,n}$.
Then the family $\{W_{K,h,n}: K\in [V_{h,n}]^{<\omega}$ and $h\in P_n\}\in \mathcal{K}_{\Omega}$ for every $n\in \mathbb{N}$.

Since $X$ satisfies $S_{\mathrm{fin}}(\mathcal{K}_\Omega, \mathcal{K}_\Omega)$, for each $n\in \mathbb{N}$ there is a finite family $T_n=\{W_{K_1,h_1,n},..., W_{K_{s(n)},h_{s(n)}}\}$ such that  $\bigcup \{T_n: n\in \mathbb{N}\}\in \mathcal{K}_\Omega$.

It remains to note that the family $\{S_n: S_n=\{h_1,...,h_{s(n)}\}\subset P_n, n\in \mathbb{N}\}$ such that $\textbf{d}(f,\bigcup S_n)\leq \varepsilon$.

\end{proof}

\begin{theorem} For every metrizable space $X$ the following conditions are equivalent:

\begin{enumerate}
\item $vet(Q_p(X,Y)) = \omega$ for all metric spaces $Y$;

\item There is a metric space $Y$ and a function $f \in Q_p(X, Y)$ such that $vet(f, Q_p(X,Y)) = \omega$;

\item There is a function $f \in Q_p(X, \mathbb{R})$ such that $vet(f, Q_p(X, \mathbb{R})) = \omega$;

\item $X$ satisfies $S_{\mathrm{fin}}(\mathcal{K}_\Omega, \mathcal{K}_\Omega)$.
\end{enumerate}
\end{theorem}

\begin{proof} $(1)\Rightarrow(2)$ and $(1)\Rightarrow(3)$ are trivial.

$(2)\Rightarrow(4)$. (i) Suppose that $vet(f, Q_p(X,Y))=\omega$ for some $f \in Q_p(X, Y)$, $f(X)\neq Y$ where $Y$ is metrizable and $|Y|>1$.

Let $\gamma_i=\{V^i_{\alpha}\}$ be a family of open subsets of $X$ such that $\gamma_i\in \mathcal{K}_\Omega$ for each $i\in \mathbb{N}$  and $y\in Y\setminus f(X)$.

For every $i\in \mathbb{N}$ and $K\in [X]^{<\omega}$ there is $V^i_{\alpha}\in \gamma_i$ such that $K=\{x_1,...,x_k\}\subset \overline{V^i_{\alpha}}$.
Since $X$ is metrizable (hence, it is open
Whyburn), there are open subsets $W_1$,...,$W_k$ of $X$ such that $x_s\in \overline{W_s}\subset \overline{V^i_{\alpha}}$  for $s=1,...,k$, $\overline{W_m}\cap \overline{W_j}=\emptyset$ for $m\neq j$.
Consider the function

 $$ f^i_{K,\alpha}:=\left\{
\begin{array}{lcl}
f(x_s) \, \, \, \, \, \, \, \, \, \, on \, \, \, \, \overline{W_s};\\
y \, \, \, \, \, \, \, \, on \, \, \, X\setminus (\bigcup\{\overline{W_s}: s=1,...,k\}).\\
\end{array}
\right.
$$

Let $F_i=\{f^i_{K,\alpha}: K\in [X]^{<\omega}\}$. It is clear that $f\in \overline{F_i}$. Then for every $i\in \mathbb{N}$
 there is a finite subset $S_i=\{f^i_{K_1,\alpha_1},..., f^i_{K_{l(i)},\alpha_{l(i)}}\}$ of $F_i$ such that $f\in \overline{\bigcup \{S_i: i\in \mathbb{N}\}}$.

Claim that $\{V^i_{\alpha_p}: p\in \{1,..., l(i)\}, i\in \mathbb{N}\}\in \mathcal{K}_\Omega$.

Let $K\in [X]^{<\omega}$. Consider the open neighborhood $[K,Y\setminus\{y\}]:=\{h\in Q_p(X,Y): h(x)\in Y\setminus\{y\}$ for every $x\in K\}$ of the point $f$ in the space $Q_p(X,Y)$. Then, there is $i\in \mathbb{N}$ and $p\in \{1,..., l(i)\}$  $f^i_{K_p,\alpha_p}\in S\cap [K,Y\setminus\{y\}]$. It follows that $K\subset \overline{V^i_{\alpha_p}}$ and, hence, $\{V^i_{\alpha_p}: p\in \{1,..., l(i)\}, i\in \mathbb{N}\}\in \mathcal{K}_\Omega$.

(ii) If $vet(f, Q_p(X,Y))=\omega$ for some $f \in Q_p(X, Y)$, $f(X)=Y$ where $Y$ is metrizable and $|Y|>1$. Let $\gamma_i=\{W^i_{\alpha}\}$ be a family of open subsets of $X$ such that $\gamma_i\in \mathcal{K}_\Omega$ for each $i\in \mathbb{N}$. Then, similar to implication $(2)\Rightarrow(4)$ in Theorem \ref{th21}, we construct the function $h_{W_{\alpha,i}, K}$ for each $W_{\alpha,i}\in \gamma_i$.

Let $F_i=\{h_{W_{\alpha,i}, K}: K\in [X]^{<\omega}\}$. It is clear that $f\in \overline{F_i}$. Then for every $i\in \mathbb{N}$
 there is a finite subset $S_i=\{h^1_{W_{\alpha_1,i}, K_1},...,h^n_{W_{\alpha_{n(i)},i}, K_{n(i)}}\}$ of $F_i$ such that $f\in \overline{\bigcup \{S_i: i\in \mathbb{N}\}}$. It is remain note that
$\{W_{\alpha_p,i}: p\in \{1,..., n(i)\}, i\in \mathbb{N}\}\in \mathcal{K}_\Omega$.

$(3)\Rightarrow(1)$. It is clear that any space satisfies $S_{\mathrm{fin}}(\mathcal{K}_\Omega, \mathcal{K}_\Omega)$ is $\mathcal{K}_\Omega$-Lindel{\"o}f. Then, by Lemma 3.1 in \cite{Osipov}, $X$ is a Lusin space.

Let $f\in Q_p(X,Y)$ and $f\in \overline{P_n}$  where $P_n\subset Q_p(X,Y)$ for each $n\in \mathbb{N}$.

Similar to the previous theorem, now we only have to consider the case $Y=\mathbb{H}$.

By Corollary \ref{cor3.5}, there is a $\mathrm{DA}$-family $\mathcal{E}$ for $f$.
For every $F\in \mathcal{E}$ denote by $\varepsilon_F:=\mathrm{dc}(f, F)$. Thus, for every $F\in \mathcal{E}$ there is $g_F\in C(F,Y)$ such that $d(f(x),g_F(x))\leq \varepsilon_F$ for every $x\in F$.

By Lemma \ref{lem5}, $F$ is $\mathcal{K}_\Omega$-Lindel{\"o}f.

Note that $d(g_F, P_n\upharpoonright F)\leq \varepsilon_F$.  By Lemma \ref{lem6.2}, there is $T_F\in [P_n]^{<\omega}$ such that $d(g_F,T_F\upharpoonright F)\leq \varepsilon_F$.

Consider $T=\bigcup \{T_F: F\in \mathcal{E}\}$. It is remain note that $f\in \overline{T}$.

$(3)\Rightarrow(4)$. It is the implication $(2)\Rightarrow(4)$ for $Y=\mathbb{R}$.

\end{proof}

\subsection{Strong fan-tightness}

Similar to the previous theorem, the following result is easy to prove.

\begin{theorem}  For every metrizable space $X$ the following conditions are equivalent:

\begin{enumerate}
\item $vet_1(Q_p(X,Y)) = \omega$ for all metric spaces $Y$;

\item There is a metric space $Y$ and a function $f \in Q_p(X, Y)$ such that $vet_1(f, Q_p(X,Y)) = \omega$;

\item There is a function $f \in Q_p(X, \mathbb{R})$ such that $vet_1(f, Q_p(X,\mathbb{R})) = \omega$;

\item $X$ satisfies $S_1(\mathcal{K}_\Omega, \mathcal{K}_\Omega)$.
\end{enumerate}
\end{theorem}

Combining the results in \cite{Os1,Osipov} and this paper, we get the following diagram in the class of metrizable spaces $X$ and $Y$.

\pagestyle{empty}

\begin{center}
\begin{tabular}{|c|c|}
\hline
$X$& $Q_p(X,\,Y)$\\
\hline
countable&Fr\'{e}chet - Urysohn\\
\hline
$S_1(K_{\Omega},\,K_{\Omega})$&$\mbox{vet}_1(Q_p(X,\,Y))=\omega$\\
\hline
$S_{fin}(K_{\Omega},\,K_{\Omega})$&$\mbox{vet}(Q_p(X,\,Y))=\omega$\\
\hline
$K_{\Omega}$-Lindel\"{o}f&$t(Q_p(X,\,Y))=\omega$\\
\hline
\end{tabular}
\end{center}

\begin{center}

%\ingrw{140}{proba_2.eps}

%\medskip
Diagram 1.

\end{center}
%\bigskip

\begin{remark} In (\cite{Osipov}, Example 6.1) it was shown that, under Jensen's axiom, there exists a metrizable Lusin space $X$ such that $vet_1(Q_p(X,\mathbb{R}))=\omega$.
An interesting open question (Questions 2 and 3 in \cite{Osipov}) remains:

{\it Is there a $T_2$-space $X$ such that $Q_p(X, \mathbb{R})$ has countable tightness  but none $Q_p(X, \mathbb{R})$ has countable (strong) fan-tightness?}
\end{remark}

%--------------------------------------------------------------------------------------------------------------------

\medskip

\bibliographystyle{model1a-num-names}
\bibliography{<your-bib-database>}

\end{document}